\documentclass[oneside]{amsart}
\usepackage[T1]{fontenc}
\usepackage[latin9]{inputenc}
\usepackage{geometry}
\usepackage{amstext}
\usepackage{amsthm}
\usepackage{amssymb}
\usepackage{bm}

\makeatletter
\numberwithin{equation}{section}
\numberwithin{figure}{section}
\theoremstyle{plain}
\newtheorem{thm}{Theorem}[section]
\newtheorem{lem}[thm]{Lemma}

\theoremstyle{definition}

\newtheorem{ex}[thm]{Example}
\theoremstyle{remark}
\newtheorem{rem}[thm]{Remark}

\usepackage{shuffle}
\usepackage[all]{xy}

\makeatother

\begin{document}
\address[Hideki Murahara]{The University of Kitakyushu,  4-2-1 Kitagata, Kokuraminami-ku, Kitakyushu, Fukuoka, 802-8577, Japan}
\email{hmurahara@mathformula.page}

\address[Tomokazu Onozuka]{Institute of Mathematics for Industry, Kyushu University 744, Motooka, Nishi-ku, Fukuoka, 819-0395, Japan} \email{t-onozuka@imi.kyushu-u.ac.jp}

\title[Asymptotic behavior of the Hurwitz-Lerch MZF at non-positive integer points]{Asymptotic behavior of the Hurwitz-Lerch multiple zeta function at non-positive integer points}

\author{Hideki Murahara}
\author{Tomokazu Onozuka}

\begin{abstract}
We give a result on the asymptotic behavior of the Hurwitz-Lerch multiple
zeta functions near non-positive integer points by using the Apostol-Bernoulli polynomials. 
From this result, we can evaluate limit values at non-positive integer points. 
\end{abstract}

\maketitle

\section{Introduction}
Let $a_{1},\dots,a_{r},z_{1},\ldots,z_{r}\in\mathbb{C}$ be parameters with $\Re(a_{1}),\Re(a_{1}+a_2),\dots,\Re(a_1+\cdots+a_{r})>0$, $|z_{1}|,\ldots,|z_{r}|\le1$, and $z_1,\ldots,z_r\neq0$. 
For $s_{1},\dots,s_{r}\in\mathbb{C}$ with $\Re(s_1),\ldots,\Re(s_r)>1$, the Hurwitz-Lerch multiple
zeta functions are defined by 
\begin{align*}
\zeta(s_{1},\ldots,s_{r};a_{1},\ldots,a_{r};z_{1},\ldots,z_{r}) & :=\sum_{0\le m_{1},\ldots,m_{r}}\frac{z_{1}^{m_{1}}\cdots z_{r}^{m_{r}}}{(m_{1}+a_{1})^{s_{1}}\cdots(m_{1}+\cdots+m_{r}+a_{1}+\cdots+a_{r})^{s_{r}}},
\end{align*}
where we set $-\pi<\arg(m_1+\cdots+m_j+a_1+\cdots+a_j)\le\pi$ for $1\le j\le r$. The function $\zeta(s;a;z)$
was studied by Lipschitz \cite{Lip57,Lip89} and Lerch \cite{Ler87}. It is known that if we fix $s$ which is not a positive integer, $\zeta(s;a;z)$ can be continued to $\mathbb{C}\setminus[1,\infty)$ as a function of $z$, and for a fixed $z\in\mathbb{C}\setminus[1,\infty)$, $\zeta(s;a;z)$ can be continued as a function of $s$ to the whole complex plain except for the possible simple poles at $s=1,2,3,\ldots$ (for details, see \cite{EMOT1}). 
Special values at non-positive integers are given by Apostol \cite{Apo51}. 
When $|z|=1$, he showed 
\begin{align*}
\zeta(-n;a;z) & =-\frac{B_{n+1}(a;z)}{n+1}
\end{align*}
for a non-negative integer $n$, where $B_{n+1}(a;z)$ is the Apostol-Bernoulli polynomial 
defined by the generating function
\[
\frac{xe^{ax}}{ze^{x}-1}=\sum_{n\ge0}B_{n}(a;z)\frac{x^{n}}{n!}.
\]
Note that $B_n(a;1)=B_n(a)$ is the Bernoulli polynomial and $(-1)^nB_n(1,1)=B_n$ is the Bernoulli number.

The Hurwitz-Lerch multiple zeta functions have been studied by many authors. 
In \cite{Kam06}, Kamano investigate the function with $z_1=\cdots=z_r=1$ and those values at non-positive integer points.
In \cite{EM20,EM21}, Essouabri and Matsumoto also studied the function that generalizes its denominators. 
In \cite{FKMT17}, Furusho, Komori, Matsumoto, and Tsumura showed some interesting analytic properties 
and defined a desingularization analogue. 
In \cite{Kom09}, Komori studied the function and gave an  analytic continuation and limit values at non-positive integer points. 
Here, motivated by the second-named author's work \cite{Onozuka13}, we shall show the asymptotic behavior of the Hurwitz-Lerch multiple zeta function at non-positive integer points. 

For  integers $i$ and $r$ with $1\le i\le r$,
put $n(i,r):=n_{i}+\cdots+n_{r}$, $l(i,r):=l_{i}+\cdots+l_{r}$,
and $\epsilon(i,r):=\epsilon_{i}+\cdots+\epsilon_{r}$. For $j=1,\dots,r-1$
and $d_{1},\dots,d_{r-1}\in\{0,1\}$, set 
\begin{align*}
S_{j}^{(d_{j})} & =S_{j}^{(d_{j})}(l_{1},\dots,l_{r})\\
 & :=\begin{cases}
\{(n_{1},\dots,n_{r})\in\mathbb{Z}_{\ge0}^{r}\mid n(j+1,r)\le l(j+1,r)+(r-j),n(1,r)=l(1,r)+r\} & \text{if }d_{j}=0,\\
\{(n_{1},\dots,n_{r})\in\mathbb{Z}_{\ge0}^{r}\mid l(j,r)+(r-j)<n(j+1,r),n(1,r)=l(1,r)+r\} & \text{if }d_{j}=1
\end{cases}
\end{align*}
and
\[
S^{(d_{1},\dots,d_{r-1})}=S^{(d_{1},\dots,d_{r-1})}(l_{1},\dots,l_{r}):=\bigcap_{j=1}^{r-1}S_{j}^{(d_{j})}.
\]
For positive integers $n_{1},\dots,n_{r}$ and $d_{1},\dots,d_{r-1}\in\{0,1\}$,
let
\begin{align*}
h^{(d_{1},\dots,d_{r-1})}(n_{1},\dots,n_{r}) & =h^{(d_{1},\dots,d_{r-1})}(n_{1},\dots,n_{r};l_{1},\dots,l_{r};\epsilon_{1},\ldots,\epsilon_{r})\\
 & :=(-1)^{l_{r}}l_{r}!\prod_{j=1}^{r-1}h_{j}^{(d_{j})}(n_{1},\dots,n_{r};l_{1},\dots,l_{r};\epsilon_{1},\ldots,\epsilon_{r}),
\end{align*}
where
\[
h_{j}^{(d_{j})}(n_{1},\dots,n_{r};l_{1},\dots,l_{r};\epsilon_{1},\ldots,\epsilon_{r}):=\begin{cases}
\displaystyle\frac{(-1)^{l_{j}}(-(n(j+1,r)-l(j,r)-(r-j)))!}{(-(n(j+1,r)-l(j+1,r)-(r-j)))!}& \text{if }d_{j}=0,\\
\displaystyle\frac{\epsilon(j+1,r)}{\epsilon(j,r)}\cdot\frac{(n(j+1,r)-l(j+1,r)-(r-j)-1)!}{(n(j+1,r)-l(j,r)-(r-j)-1)!}& \text{if }d_{j}=1.
\end{cases}
\]
In this paper, we present asymptotic behavior of the Hurwitz-Lerch multiple
zeta function at non-positive integer points. 
\begin{thm}\label{main}
Let $r\ge2$ and $a_{1},\dots,a_{r},z_{1},\ldots,z_{r}$ be complex parameters with $\Re(a_{1}),\Re(a_{1}+a_2),\dots,\Re(a_1+\cdots+a_{r})>0$ and $z_{1},\ldots,z_{r}\notin(1,\infty)$, 
and $\epsilon_{1},\ldots,\epsilon_{r}$ complex numbers. Suppose that $\left|\epsilon_{1}\right|,\ldots,\left|\epsilon_{r}\right|$ are sufficiently small with $\epsilon_{j}\neq0,\epsilon(j,r)\neq0$,
and $\left|\epsilon_{k}/\epsilon(j,r)\right|\ll1$ as $(\epsilon_{1},\ldots,\epsilon_{r})\rightarrow(0,\ldots,0)$
for $j=1,\ldots,r$ and $k=j,\ldots,r$. For non-negative integers
$l_{1},\dots,l_{r}$, 
we have 
\begin{align*}
 &\zeta(-l_{1}+\epsilon_{1},\dots,-l_{r}+\epsilon_{r};a_{1,}\ldots,a_{r};z_{1},\ldots,z_{r})\\
 &=(-1)^{l(1,r)+r}\sum_{d_{1},\dots,d_{r-1}\in\{0,1\}}\sum_{(n_{1},\dots,n_{r})\in S^{(d_{1},\dots,d_{r-1})}}\frac{B_{n_{1}}(a_1;z_1)\cdots B_{n_{r}}(a_{r};z_{r})}{n_{1}!\cdots n_{r}!}h^{(d_{1},\dots,d_{r-1})}(n_{1},\dots,n_{r})\\
 & \quad+\sum_{j=1}^{r}O(|\epsilon_{j}|).
\end{align*}
\end{thm}
\begin{rem}
When $a_1=\cdots=a_r=1$ and $z_1=\cdots=z_r=1$, we can obtain \cite[Theorem 2]{Onozuka13}.
\end{rem}
Some examples are given as follows. 
\begin{ex}
When $r=2$ and $(l_1,l_2)=(0,0)$, by Theorem \ref{main}, we have
\begin{align*}
\zeta(\epsilon_1,\epsilon_2;a_1,a_2;z_1,z_2) 
&=B_{1}(a_1;z_1) B_{1}(a_{2};z_{2}) +\frac{1}{2} B_{2}(a_1;z_1) B_{0}(a_{2};z_{2})\\
&\quad +\frac{1}{2} B_{0}(a_1;z_1) B_{2}(a_{2};z_{2}) \frac{\epsilon_{2}}{\epsilon_{1}+\epsilon_{2}}+\sum_{j=1}^{2}O(|\epsilon_{j}|).
\end{align*}
When $z_1,z_2\neq1$, since we can easily obtain $B_{0}(a;z) =0$ and $B_{1}(a;z) =(z-1)^{-1}$, we have
\begin{align*}
\zeta(\epsilon_1,\epsilon_2;a_1,a_2;z_1,z_2) 
&=(z_1-1)^{-1}(z_2-1)^{-1} +\sum_{j=1}^{2}O(|\epsilon_{j}|).
\end{align*}
When $z_1=1,z_2\neq1$,  since $B_{n}(a;1) $ is the Bernoulli polynomial and $B_{2}(a;z) =2a/(z-1)-2 z/(z-1)^2$, we have
\begin{align*}
\zeta(\epsilon_1,\epsilon_2;a_1,a_2;1,z_2) 
&=\left(a_{1}-\frac{1}{2}\right) \frac{1}{z_{2}-1} +\left(\frac{1}{z_2-1}a_2-\frac{ z_2}{(z_2-1)^2}\right)\frac{\epsilon_{2}}{\epsilon_{1}+\epsilon_{2}}+\sum_{j=1}^{2}O(|\epsilon_{j}|).
\end{align*}
When $z_1\neq1,z_2=1$ and $z_1=z_2=1$ ,  we have
\begin{align*}
&\zeta(\epsilon_1,\epsilon_2;a_1,a_2;z_1,1) 
=\frac{1}{z_{1}-1} \left(a_1+a_{2}-\frac{3}{2}\right)  -\frac{ 1}{(z_1-1)^2}+\sum_{j=1}^{2}O(|\epsilon_{j}|),\\
&\zeta(\epsilon_1,\epsilon_2;a_1,a_2;1,1) 
=\left(a_{1}-\frac{1}{2}\right) \left(a_{2}-\frac{1}{2}\right)  +\frac{1}{2}\left(a_1^2-a_1+\frac{1}{6}\right)+\frac{1}{2}\left(a_2^2-a_2+\frac{1}{6}\right)\frac{\epsilon_{2}}{\epsilon_{1}+\epsilon_{2}}+\sum_{j=1}^{2}O(|\epsilon_{j}|),
\end{align*}
respectively. In the last section, we give other examples.
\end{ex}

\begin{rem}
 Some of the special cases of the limit values were also studied by Akiyama and Tanigawa \cite{AT01}, Komori \cite{Kom09},  and Sasaki \cite{Sas09}.
\end{rem}

\section{preliminaries}
We give some results on the Apostol-Bernoulli polynomial.
Let $\lambda(z)$ be one of the poles of the generating function $xe^{ax}/(ze^{x}-1)$
closest to $x=0$ except for $x=0$, and $S(n,k)$ the Stirling number of the second kind.
(For $z=0$, we define $\lambda(0)=\infty$.)
\begin{lem}\label{apos}
For  $|x|<|\lambda(z)|$, we have
$$\frac{xe^{-ax}}{1-ze^{-x}}=\sum_{n\ge0}(-1)^{n}B_{n}(a;z)\frac{x^{n}}{n!}.$$
\end{lem}
\begin{proof}
By the definition of the Apostol-Bernoulli polynomial, we have
\begin{align*}
\frac{xe^{-ax}}{1-ze^{-x}}&=\frac{(-x)e^{a(-x)}}{ze^{-x}-1}
=\sum_{n\ge0}(-1)^{n}B_{n}(a;z)\frac{x^{n}}{n!}.\qedhere
\end{align*}
\end{proof}
\begin{thm}[Apostol \cite{Apo51}]\label{a51}
For a complex number $z\neq1$, the Apostol-Bernoulli polynomial  $B_n(a;z)$ can be written by 
$$
B_n(a;z)=\sum_{k=0}^n\binom{n}{k}\beta_k(z)a^{n-k},
$$
where $\beta_k(z)=B_k(0;z)$ is defined by
$$
\beta_k(z)=\frac{k}{(z-1)^k}\sum_{l=0}^{k-1}l!(-z)^{l}(z-1)^{k-1-l}S(k-1,l).
$$
\end{thm}
\begin{lem}\label{estimate2}
Let $z\neq1$. 
For a positive integer $n$, we have
\begin{align*}
|B_{n}(a;z)| \ll\frac{n!}{(\log2)^n}e^{(\log2)|a|}\max\left\{|z|^{n-1}|z-1|^{-n},|z-1|^{-1}\right\},
\end{align*}
where the implicit constant does not depend on $a$, $z$, and $n$.
\end{lem}
\begin{proof}
Note that the estimate of the ordered Bell number is known as $\sum_{l=0}^{k-1}l!S(k-1,l)\ll(k-1)!/(\log2)^k$.
By the previous theorem, we have
\begin{align*}
\beta_k(z)&\le\frac{k}{|z-1|^k}\max\left\{|z|^{k-1},|z-1|^{k-1}\right\}\sum_{l=0}^{k-1}l!S(k-1,l)\\
&\ll \frac{k!}{(\log 2)^k|z-1|^k}\max\left\{|z|^{k-1},|z-1|^{k-1}\right\}.
\end{align*}
 Hence we have
\begin{align*}
|B_n(a;z)|&\ll\sum_{k=1}^n\binom{n}{k}\frac{k!}{(\log 2)^k|z-1|^k}\max\left\{|z|^{k-1},|z-1|^{k-1}\right\}|a|^{n-k}\\
&\ll n!\sum_{k=1}^n\frac{1}{(n-k)!}\frac{|a|^{n-k}}{(\log 2)^k|z-1|^k}\max\left\{|z|^{k-1},|z-1|^{k-1}\right\}\\
&\ll \frac{n!}{(\log2)^n}e^{(\log2)|a|}\max\left\{|z|^{n-1}|z-1|^{-n},|z-1|^{-1}\right\}.\qedhere
\end{align*}
\end{proof}
\begin{lem}\label{estimate3} 
For a positive integer $n$, we have
\begin{align*}
|B_{n}(a;1)| \ll\frac{n!}{(2\pi)^{n}}e^{2\pi|a|},
\end{align*}
where the implicit constant does not depend on $a$ and $n$.
\end{lem}
\begin{proof}
Since $|B_n|\ll n!/(2\pi)^n$, we have
\begin{align*}
|B_n(a;1)|&\le\sum_{k=0}^n\binom{n}{k}|B_{n-k}||a|^k\\
&\ll\sum_{k=0}^n\frac{n!}{k!}\frac{|a|^k}{(2\pi)^{n-k}}\\
&\ll\frac{n!}{(2\pi)^{n}}e^{2\pi|a|}.\qedhere
\end{align*}
\end{proof}

\section{Meromorphic Continuation}

By the definition of the gamma function,  for $\Re(s_1),\ldots,\Re(s_r)>1$, $\Re(a_{1}),\Re(a_{1}+a_2),\dots,\Re(a_1+\cdots+a_{r})>0$, and $|z_1|,\ldots,|z_r|\le1$ with $z_1,\ldots,z_r\neq0$, we have
\begin{align*}
&\Gamma(s_1)\cdots\Gamma(s_r)\frac{z_{1}^{m_{1}}\cdots z_{r}^{m_{r}}}{(m_{1}+a_{1})^{s_{1}}\cdots(m_{1}+\cdots+m_{r}+a_{1}+\cdots+a_{r})^{s_{r}}}\\
&=\int_0^\infty \cdots \int_0^\infty e^{-(m_1+a_1)u_1-\cdots-(m_1+\cdots+m_r+a_1+\cdots+a_r)u_r}z_{1}^{m_{1}}\cdots z_{r}^{m_{r}}u_1^{s_1-1}\cdots u_r^{s_r-1}du_1\cdots du_r.
\end{align*}
By taking the sum of the both sides, we have
\begin{align*}
 & \Gamma(s_{1})\cdots\Gamma(s_{r})\zeta(s_{1},\ldots,s_{r};a_{1,}\ldots,a_{r};z_{1},\ldots,z_{r})\\
 & =\int_{0}^{\infty}\cdots\int_{0}^{\infty}\frac{e^{-a_{1}(u_{1}+\cdots+u_{r})}}{1-z_{1}e^{-(u_{1}+\cdots+u_{r})}}\cdot\frac{e^{-a_{2}(u_{2}+\cdots+u_{r})}}{1-z_{2}e^{-(u_{2}+\cdots+u_{r})}}\cdot\cdots\cdot\frac{e^{-a_{r}u_{r}}}{1-z_{r}e^{-u_{r}}}u_{1}^{s_{1}-1}\cdots u_{r}^{s_{r}-1}du_{1}\cdots du_{r}.
\end{align*}
Here, we use a change of variables
$$x_1\cdots x_j=u_j+\cdots+u_r\iff u_j=x_1\cdots x_j(1-x_{j+1})$$
for $j=1,\ldots,r$, where $x_{r+1}=0$. Since the Jacobian is $x_1^{r-1}\cdots x_{r-1}$, we have
\begin{align*} 
 & \Gamma(s_{1})\cdots\Gamma(s_{r})\zeta(s_{1},\ldots,s_{r};a_{1,}\ldots,a_{r};z_{1},\ldots,z_{r})\\
 & =\int_{0}^{1}\cdots\int_{0}^{1}\int_{0}^{\infty}\frac{e^{-a_{1}x_{1}}}{1-z_{1}e^{-x_{1}}}\cdot\frac{e^{-a_{2}x_{1}x_{2}}}{1-z_{2}e^{-x_{1}x_{2}}}\cdot\cdots\cdot\frac{e^{-a_{r}x_{1}\cdots x_{r}}}{1-z_{r}e^{-x_{1}\cdots x_{r}}}\\
 & \quad\quad(x_{1}(1-x_{2}))^{s_{1}-1}(x_{1}x_{2}(1-x_{3}))^{s_{2}-1}\cdots(x_{1}\cdots x_{r-1}(1-x_{r}))^{s_{r-1}-1}(x_{1}\cdots x_{r})^{s_{r}-1}\\
 & \quad\quad x_{1}^{r-1}\cdots x_{r-1}dx_{1}\cdots dx_{r}\\
 & =\int_{0}^{1}\cdots\int_{0}^{1}\int_{0}^{\infty}x_{1}^{s_{1}+\cdots+s_{r}-r-1}x_{2}^{s_{2}+\cdots+s_{r}-(r-1)-1}\cdots x_{r}^{s_{r}-2}\\
 & \quad\quad(1-x_{2}){}^{s_{1}-1}\cdots(1-x_{r}){}^{s_{r-1}-1}\\
 & \quad\quad\frac{x_1e^{-a_{1}x_{1}}}{1-z_{1}e^{-x_{1}}}\cdot\frac{x_1x_2e^{-a_{2}x_{1}x_{2}}}{1-z_{2}e^{-x_{1}x_{2}}}\cdot\cdots\cdot\frac{x_1x_2\cdots x_re^{-a_{r}x_{1}\cdots x_{r}}}{1-z_{r}e^{-x_{1}\cdots x_{r}}}dx_{1}\cdots dx_{r}.
\end{align*}

Fix a small positive number $c$ with $0<c<\min_{1\le j\le r}\{|\lambda(z_j)|\}$. Put 
\begin{align*}
X_{1} & :=\int_{0}^{1}\cdots\int_{0}^{1}\int_{0}^{c}x_{1}^{s_{1}+\cdots+s_{r}-r-1}x_{2}^{s_{2}+\cdots+s_{r}-(r-1)-1}\cdots x_{r}^{s_{r}-2}\\
 & \quad\quad(1-x_{2}){}^{s_{1}-1}\cdots(1-x_{r}){}^{s_{r-1}-1}\\
 & \quad\quad\frac{x_1e^{-a_{1}x_{1}}}{1-z_{1}e^{-x_{1}}}\cdot\frac{x_1x_2e^{-a_{2}x_{1}x_{2}}}{1-z_{2}e^{-x_{1}x_{2}}}\cdot\cdots\cdot\frac{x_1x_2\cdots x_re^{-a_{r}x_{1}\cdots x_{r}}}{1-z_{r}e^{-x_{1}\cdots x_{r}}}dx_{1}\cdots dx_{r},\\
X_{2} & :=\int_{0}^{1}\cdots\int_{0}^{1}\int_{c}^{\infty}x_{1}^{s_{1}+\cdots+s_{r}-r-1}x_{2}^{s_{2}+\cdots+s_{r}-(r-1)-1}\cdots x_{r}^{s_{r}-2}\\
 & \quad\quad(1-x_{2}){}^{s_{1}-1}\cdots(1-x_{r}){}^{s_{r-1}-1}\\
 & \quad\quad\frac{x_1e^{-a_{1}x_{1}}}{1-z_{1}e^{-x_{1}}}\cdot\frac{x_1x_2e^{-a_{2}x_{1}x_{2}}}{1-z_{2}e^{-x_{1}x_{2}}}\cdot\cdots\cdot\frac{x_1x_2\cdots x_re^{-a_{r}x_{1}\cdots x_{r}}}{1-z_{r}e^{-x_{1}\cdots x_{r}}}dx_{1}\cdots dx_{r}.
\end{align*}
Then we have
\begin{align*}
 & \Gamma(s_{1})\cdots\Gamma(s_{r})\zeta(s_{1},\ldots,s_{r};a_{1,}\ldots,a_{r};z_{1},\ldots,z_{r})\\
 & =X_{1}+X_{2}.
\end{align*}
By Lemma \ref{apos}, we have 
\begin{align*}
X_{1}& =\int_{0}^{1}\cdots\int_{0}^{1}\int_{0}^{c}x_{1}^{s_{1}+\cdots+s_{r}-r-1}x_{2}^{s_{2}+\cdots+s_{r}-(r-1)-1}\cdots x_{r}^{s_{r}-2}\protect\\
 & \quad\quad(1-x_{2}){}^{s_{1}-1}\cdots(1-x_{r}){}^{s_{r-1}-1}\protect\\
 & \quad\quad\biggl(\sum_{n_{1}\ge0}(-1)^{n_1}B_{n_{1}}(a_1;z_1)\frac{x_{1}^{n_{1}}}{n_{1}!}\biggr)\cdot\biggl(\sum_{n_{2}\ge0}(-1)^{n_2}B_{n_{2}}(a_{2};z_{2})\frac{(x_{1}x_{2})^{n_{2}}}{n_{2}!}\biggr)\cdots\protect\\
 & \quad\quad\biggl(\sum_{n_{r}\ge0}(-1)^{n_r}B_{n_{r}}(a_{r};z_{r})\frac{(x_{1}\cdots x_{r})^{n_{r}}}{n_{r}!}\biggr)dx_{1}\cdots dx_{r}.
\end{align*}
Thus we find
\begin{align*}
X_{1} & =\int_{0}^{1}\cdots\int_{0}^{1}\int_{0}^{c}x_{1}^{s_{1}+\cdots+s_{r}-r-1}x_{2}^{s_{2}+\cdots+s_{r}-(r-1)-1}\cdots x_{r}^{s_{r}-2}\\
 & \quad\quad(1-x_{2}){}^{s_{1}-1}\cdots(1-x_{r}){}^{s_{r-1}-1}\\
 & \quad\quad\sum_{k\ge0}\sum_{n_{1}+\cdots+n_{r}=k,n_{1},\dots,n_{r}\ge0}(-1)^k\frac{B_{n_{1}}(a_1;z_1)\cdots B_{n_{r}}(a_{r};z_{r})}{n_{1}!\cdots n_{r}!}x_{1}^{k}x_{2}^{n_{2}+\cdots+n_{r}}\cdots x_{r}^{n_{r}}dx_{1}\cdots dx_{r}.
\end{align*}
\\
For a non-negative integer $N$, put 
\begin{align*}
Y_{1} & :=\int_{0}^{1}\cdots\int_{0}^{1}\int_{0}^{c}x_{1}^{s_{1}+\cdots+s_{r}-r-1}x_{2}^{s_{2}+\cdots+s_{r}-(r-1)-1}\cdots x_{r}^{s_{r}-2}\times(1-x_{2}){}^{s_{1}-1}\cdots(1-x_{r}){}^{s_{r-1}-1}\\
 & \quad\quad\sum_{0\le k\le N}\sum_{n_{1}+\cdots+n_{r}=k,n_{1},\dots,n_{r}\ge0}(-1)^k\frac{B_{n_{1}}(a_1;z_1)\cdots B_{n_{r}}(a_{r};z_{r})}{n_{1}!\cdots n_{r}!}x_{1}^{k}x_{2}^{n_{2}+\cdots+n_{r}}\cdots x_{r}^{n_{r}}dx_{1}\cdots dx_{r},\\
Y_{2} & :=\int_{0}^{1}\cdots\int_{0}^{1}\int_{0}^{c}x_{1}^{s_{1}+\cdots+s_{r}-r-1}x_{2}^{s_{2}+\cdots+s_{r}-(r-1)-1}\cdots x_{r}^{s_{r}-2}\times(1-x_{2}){}^{s_{1}-1}\cdots(1-x_{r}){}^{s_{r-1}-1}\\
 & \quad\quad\sum_{k>N}\sum_{n_{1}+\cdots+n_{r}=k,n_{1},\dots,n_{r}\ge0}(-1)^k\frac{B_{n_{1}}(a_1;z_1)\cdots B_{n_{r}}(a_{r};z_{r})}{n_{1}!\cdots n_{r}!}x_{1}^{k}x_{2}^{n_{2}+\cdots+n_{r}}\cdots x_{r}^{n_{r}}dx_{1}\cdots dx_{r}.
\end{align*}
Note that $X_{1}=Y_{1}+Y_{2}$. By changing
the order of sums and integrals of $Y_1$, we have 
\begin{align}
Y_{1} & =\sum_{0\le k\le N}\sum_{n_{1}+\cdots+n_{r}=k,n_{1},\dots,n_{r}\ge0}(-1)^k\frac{B_{n_{1}}(a_1;z_1)\cdots B_{n_{r}}(a_{r};z_{r})}{n_{1}!\cdots n_{r}!}\int_{0}^{c}x_{1}^{k+s_{1}+\cdots+s_{r}-r-1}dx_{1}\nonumber \\
 & \quad\int_{0}^{1}x_{2}^{n_{2}+\cdots+n_{r}+s_{2}+\cdots+s_{r}-(r-1)-1}(1-x_{2}){}^{s_{1}-1}dx_{2}\nonumber \\
 & \quad\int_{0}^{1}x_{3}^{n_{3}+\cdots+n_{r}+s_{3}+\cdots+s_{r}-(r-2)-1}(1-x_{3}){}^{s_{2}-1}dx_{3}\cdots\nonumber \\
 & \quad\int_{0}^{1}x_{r}^{n_{r}+s_{r}-2}(1-x_{r}){}^{s_{r-1}-1}dx_{r}\nonumber \\
 & =\sum_{0\le k\le N}\sum_{n_{1}+\cdots+n_{r}=k,n_{1},\dots,n_{r}\ge0}(-1)^k\frac{B_{n_{1}}(a_1;z_1)\cdots B_{n_{r}}(a_{r};z_{r})}{n_{1}!\cdots n_{r}!}\label{Y1}
 \cdot\frac{c^{k+s_{1}+\cdots+s_{r}-r}}{k+s_{1}+\cdots+s_{r}-r} \\
 & \quad\cdot\frac{\Gamma(n_{2}+\cdots+n_{r}+s_{2}+\cdots+s_{r}-(r-1))\Gamma(s_{1})}{\Gamma(n_{2}+\cdots+n_{r}+s_{1}+s_{2}+\cdots+s_{r}-(r-1))}\nonumber \\
 & \quad\cdot\frac{\Gamma(n_{3}+\cdots+n_{r}+s_{3}+\cdots+s_{r}-(r-2))\Gamma(s_{2})}{\Gamma(n_{3}+\cdots+n_{r}+s_{2}+\cdots+s_{r}-(r-2))}\cdots\nonumber \\
 & \quad\cdot\frac{\Gamma(n_{r}+s_{r}-1)\Gamma(s_{r-1})}{\Gamma(n_{r}+s_{r-1}+s_{r}-1)}.\nonumber 
\end{align}
When $z_1,\ldots,z_r\neq1$, by Theorem \ref{a51}, $Y_{1}$, as a function of $(z_1,\ldots,z_r)$, can be continued to
$$
\{(z_1,\ldots,z_r)\in\mathbb{C}^r\mid z_1\neq1,\ldots ,z_r\neq1\}.
$$
Hence, $Y_{1}$ can be continued meromorphically to $\mathbb{C}^{2r}$ as a function of $(s_1,\ldots,s_r,z_1,\ldots,z_r)$. Generally, when $z_{p_1},\ldots,z_{p_{\xi}}\neq1$ and $z_{q_1}=\cdots=z_{q_{r-\xi}}=1$, $Y_{1}$ can be continued meromorphically to $\mathbb{C}^{r+\xi}$ as a function of $(s_1,\ldots,s_r,z_{p_1},\ldots,z_{p_{\xi}})$ since $B_n(a;1)$ is the Bernoulli polynomial.

Now, we consider $Y_{2}$. By changing the exponent of $x_1$, we have 
\begin{align*}
Y_{2} & =\int_{0}^{1}\cdots\int_{0}^{1}\int_{0}^{c}x_{1}^{s_{1}+\cdots+s_{r}-r+N}x_{2}^{s_{2}+\cdots+s_{r}-(r-1)-1}\cdots x_{r}^{s_{r}-2}\cdot(1-x_{2}){}^{s_{1}-1}\cdots(1-x_{r}){}^{s_{r-1}-1}\\
 & \quad\quad\sum_{k>N}\sum_{n_{1}+\cdots+n_{r}=k,n_{1},\dots,n_{r}\ge0}(-1)^k\frac{B_{n_{1}}(a_1;z_1)\cdots B_{n_{r}}(a_{r};z_{r})}{n_{1}!\cdots n_{r}!}x_{1}^{k-N-1}x_{2}^{n_{2}+\cdots+n_{r}}\cdots x_{r}^{n_{r}}dx_{1}\cdots dx_{r}.
\end{align*}
We consider the convergence of the sum of the last line. 
For $z_1,\ldots,z_r\neq1$, by Lemma
 \ref{estimate2}, we have\\
\begin{align}
 & \sum_{k>N}\sum_{n_{1}+\cdots+n_{r}=k,n_{1},\dots,n_{r}\ge0}\biggl|\frac{B_{n_{1}}(a_1;z_1)\cdots B_{n_{r}}(a_{r};z_{r})}{n_{1}!\cdots n_{r}!}\biggr|\cdot|x_{1}^{k-N-1}x_{2}^{n_{2}+\cdots+n_{r}}\cdots x_{r}^{n_{r}}|\label{2.3}\\ 
 & \ll c^{-N-1}e^{(\log2)(|a_1|+\cdots+|a_r|)}\nonumber\\
 &\quad\sum_{n_{1}\ge1}\left(\frac{\max\left\{|z_1|^{n_1-1}|z_1-1|^{-n_1},|z_1-1|^{-1}\right\}c^{n_1}}{(\log2)^{n_1}}\right)\cdots\sum_{n_{r}\ge1}\left(\frac{\max\left\{|z_r|^{n_r-1}|z_r-1|^{-n_r},|z_r-1|^{-1}\right\}c^{n_r}}{(\log2)^{n_r}}\right).\nonumber
\end{align}
Hence, by replacing $c>0$ with a smaller one if necessary, the above series converges in $[0,c]\times[0,1]^{r-1}$. 
Moreover, for a small $c>0$, the above series converges uniformly as a function of $(z_1,\ldots,z_r)$  in
$$
T_c:=\{(z_1,\ldots,z_r)\in\mathbb{C}^r\mid |z_j-1|>2\sqrt{c},\ |z_j|<(\log2)/\sqrt{c}\quad(j=1,\ldots,r)\}
$$
since $|z-1|>2\sqrt{c}$ and $|z|<(\log2)/\sqrt{c}$ yield $c|z|/(|z-1|\log2)<1/2$. Thus, series \eqref{2.3} can be continued to $T_c$.
Similarly, when $z_{p_1},\ldots,z_{p_{\xi}}\neq1$ and $z_{q_1}=\cdots=z_{q_{r-\xi}}=1$,  series \eqref{2.3} can be continued to some region contained in $(\mathbb{C}\setminus\{1\})^\xi$ as a function of $(z_{p_1},\ldots,z_{p_\xi})$ by Lemma \ref{estimate3}. 

Put
\[
F(x_{1},\dots,x_{r}):=\sum_{k>N}\sum_{n_{1}+\cdots+n_{r}=k,n_{1},\dots,n_{r}\ge0}(-1)^k\frac{B_{n_{1}}(a_1;z_1)\cdots B_{n_{r}}(a_{r};z_{r})}{n_{1}!\cdots n_{r}!}x_{1}^{k-N-1}x_{2}^{n_{2}+\cdots+n_{r}}\cdots x_{r}^{n_{r}}.
\]
From the above arguments, we see that $F(x_{1},\dots,x_{r})$ is holomorphic function on some region
containing $[0,c]\times[0,1]^{r-1}$. 
For $i_{2},\dots,i_{r}\in\{0,1/2\}$, put
\begin{align*}
\widetilde{Y}_{2}(i_{2},\dots,i_{r}) & :=\int_{0+i_{r}}^{1/2+i_{r}}\cdots\int_{0+i_{2}}^{1/2+i_{2}}\int_{0}^{c}x_{1}^{s_{1}+\cdots+s_{r}-r+N}x_{2}^{s_{2}+\cdots+s_{r}-(r-1)-1}\cdots x_{r}^{s_{r}-2}\\
 & \quad\quad(1-x_{2}){}^{s_{1}-1}\cdots(1-x_{r}){}^{s_{r-1}-1}F(x_{1},\dots,x_{r})dx_{1}\cdots dx_{r}.
\end{align*}
Note that 
\[
Y_{2}=\sum_{i_{2},\dots,i_{r}\in\{0,1/2\}}\widetilde{Y}_{2}(i_{2},\dots,i_{r}).
\]
If $i_{j}=0$ for all $j=2,\dots,r$, by repeating the integration by parts,
we have
\begin{align}
 & \int_{0}^{1/2}x_{j}^{s_{j}+\cdots+s_{r}-(r-j)-1}(1-x_{j}){}^{s_{j-1}-1}G(x_{1},\dots,x_{r})dx_{j}\nonumber \\
 & =\sum_{l=0}^{n}(-1)^{l}\frac{(1/2)^{s_{j}+\cdots+s_{r}-(r-j)+l}}{(s_{j}+\cdots+s_{r}-(r-j))_{l+1}}\cdot\left[\frac{d^{l}}{dx_{j}^{l}}((1-x_{j}){}^{s_{j-1}-1}G(x_{1},\dots,x_{r}))\right]_{x_{j}=1/2}\label{integration by part1}\\
 & \quad+(-1)^{n+1}\int_{0}^{1/2}\frac{x_{j}^{s_{j}+\cdots+s_{r}-(r-j)+n}}{(s_{j}+\cdots+s_{r}-(r-j))_{n+1}}\cdot\frac{d^{n+1}}{dx_{j}^{n+1}}((1-x_{j}){}^{s_{j-1}-1}G(x_{1},\dots,x_{r}))dx_{j},\nonumber 
\end{align}
where $G(x_{1},\dots,x_{r})$ be a holomorphic function on some region
containing $[0,c]\times[0,1]^{r-1}$. The last integral converges
if $\Re(s_{j}+\cdots+s_{r}-(r-j)+n)>-1$. The right-hand side is meromorphic
on the same region and the possible poles are simple poles located
on $s_{j}+\cdots+s_{r}=r-j-p\quad(p=0,\dots,n)$. Similarly, if $i_{j}=1/2$,
we also have
\begin{align}
 & \int_{1/2}^{1}x_{j}^{s_{j}+\cdots+s_{r}-(r-j)-1}(1-x_{j}){}^{s_{j-1}-1}G(x_{1},\dots,x_{r})dx_{j}\nonumber \\
 & =\sum_{l=0}^{n}\frac{(1/2)^{s_{j-1}+l}}{(s_{j-1})_{l+1}}\cdot\left[\frac{d^{l}}{dx_{j}^{l}}(x_{j}^{s_{j}+\cdots+s_{r}-(r-j)-1}G(x_{1},\dots,x_{r}))\right]_{x_{j}=1/2}\label{integration by part2}\\
 & \quad+\int_{1/2}^{1}\frac{(1-x_{j})^{s_{j-1}+n}}{(s_{j-1})_{n+1}}\cdot\frac{d^{n+1}}{dx_{j}^{n+1}}((1-x_{j}){}^{s_{j-1}-1}G(x_{1},\dots,x_{r}))dx_{j}.\nonumber 
\end{align}
The last integral converges if $\Re(s_{j-1}+n)>-1$. The right-hand
side is meromorphic on the same region and the possible poles are
simple poles at $s_{j-1}=-p\quad(p=0,\dots,n)$. Thus, for $i_{2},\dots,i_{r}\in\{0,1/2\}$,
we see that the function $\widetilde{Y}_{2}(i_{2},\dots,i_{r})$ can
be continued meromorphically to the region 
\begin{align*}
\{(s_{1},\dots,s_{r},z_{p_1},\ldots,z_{p_{\xi}})\mid&\Re(s_{j}+\cdots+s_{r}-(r-j)+n)>-1,\Re(s_{j-1}+n)>-1\quad(j=2,\dots,r),\\
&\Re(s_{1}+\cdots+s_{r}-r+N)>-1,\\
&|\lambda(z_{p_j})|>c,\ |z_{p_j}-1|>2\sqrt{c},\ |z_{p_j}|<(\log2)/\sqrt{c}\quad(j=1,\dots,\xi)\}
\end{align*}
when $z_{p_1},\ldots,z_{p_{\xi}}\neq1$ and $z_{q_1}=\cdots=z_{q_{r-\xi}}=1$.
For $i_{2},\dots,i_{r}\in\{0,1/2\}$, let
\[
U_{j}^{p}=U_{j}^{p}(i_{2},\dots,i_{r}):=\begin{cases}
\{(s_{1},\dots,s_{r})\mid s_{j}+\cdots+s_{r}=r-j-p\} & \text{if }i_{j}=0,\\
\{(s_{1},\dots,s_{r})\mid s_{j-1}=-p\} & \text{if }i_{j}=1/2.
\end{cases}
\]
The possible poles as a function of $(s_{1},\dots,s_{r})$ are simple poles located on the region
\[
\bigcup_{p=0}^{n}\bigcup_{j=2}^{r}U_{j}^p.
\]

Now we consider $X_{2}$. When $z_{p_1},\ldots,z_{p_{\xi}}\neq1$ and $z_{q_1}=\cdots=z_{q_{r-\xi}}=1$, we can easily check that $X_2$ can be continued to $(\mathbb{C}\setminus[1,\infty))^\xi$ as a function of $(z_{p_1},\ldots,z_{p_\xi})$. 
For $i_{2},\dots,i_{r}\in\{0,1/2\}$, put
\begin{align*}
\widetilde{X}_{2}(i_{2},\dots,i_{r}) & :=\int_{0+i_{r}}^{1/2+i_{r}}\cdots\int_{0+i_{2}}^{1/2+i_{2}}\int_{c}^{\infty}x_{1}^{s_{1}+\cdots+s_{r}-r-1}x_{2}^{s_{2}+\cdots+s_{r}-(r-1)-1}\cdots x_{r}^{s_{r}-2}\\
 & \quad(1-x_{2}){}^{s_{1}-1}\cdots(1-x_{r}){}^{s_{r-1}-1}\\
 & \quad\frac{x_1e^{-a_{1}x_{1}}}{1-z_{1}e^{-x_{1}}}\cdot\frac{x_1x_2e^{-a_{2}x_{1}x_{2}}}{1-z_{2}e^{-x_{1}x_{2}}}\cdot\cdots\cdot\frac{x_1x_2\cdots x_re^{-a_{r}x_{1}\cdots x_{r}}}{1-z_{r}e^{-x_{1}\cdots x_{r}}}dx_{1}\cdots dx_{r}.
\end{align*}
Note that 
\[
X_{2}=\sum_{i_{2},\dots,i_{r}\in\{0,1/2\}}\widetilde{X}_{2}(i_{2},\dots,i_{r}).
\]

Now we reset 
\[
F(x_{1},\dots,x_{r}):=\frac{x_1e^{-a_{1}x_{1}}}{1-z_{1}e^{-x_{1}}}\cdot\frac{x_1x_2e^{-a_{2}x_{1}x_{2}}}{1-z_{2}e^{-x_{1}x_{2}}}\cdot\cdots\cdot\frac{x_1x_2\cdots x_re^{-a_{r}x_{1}\cdots x_{r}}}{1-z_{r}e^{-x_{1}\cdots x_{r}}}.
\]
Then for $z_1,\ldots,z_r\in\mathbb{C}\setminus(1,\infty)$, we see that the function $F(x_{1},\dots,x_{r})$ is holomorphic
function on some region containing $[c,\infty)\times[0,1]^{r-1}$.
Since $\Re(a_1+\cdots+a_j)>0$ for all $j=1,\ldots,r$ and
\begin{align*}
&-a_{1}x_{1}-a_{2}x_{1}x_{2}-\cdots-a_{r}x_{1}\cdots x_{r}\\
&=-a_1x_1(1-x_2)-\cdots-(a_1+\cdots+a_{r-1})x_1\cdots x_{r-1}(1-x_r)-(a_1+\cdots+a_{r})x_1\cdots x_{r},
\end{align*}
the functions 
\[
\left|\Bigl(\frac{d}{dx_{2}}\Bigr)^{n_{2}}\cdots\Bigl(\frac{d}{dx_{r}}\Bigr)^{n_{r}}F(x_{1},\dots,x_{r})\right|
\]
decrease rapidly on $[c,\infty)\times[0,1]^{r-1}$ if $x_{1}$ tends
to $\infty$. From the similar argument in \eqref{integration by part1}
and \eqref{integration by part2}, we see that $\widetilde{X}_{2}(i_{2},\dots,i_{r})$
is meromorphic on 
\begin{align*}
\{(s_{1},\dots,s_{r},z_{p_1},\ldots,z_{p_{\xi}})\mid&\Re(s_{j}+\cdots+s_{r}-(r-j)+n)>-1,\Re(s_{j-1}+n)>-1\quad(j=2,\dots,r),\\
&z_{p_j}\in\mathbb{C}\setminus[1,\infty)\quad(j=1,\dots,\xi)\}
\end{align*}
when $z_{p_1},\ldots,z_{p_{\xi}}\neq1$ and $z_{q_1}=\cdots=z_{q_{r-\xi}}=1$.
The possible poles as a function of $(s_{1},\dots,s_{r})$ are simple poles located on the region 
\[
\bigcup_{p=0}^{n}\bigcup_{j=2}^{r}U_{j}^p.
\]

\section{Proof of Theorem \ref{main}}

In the previous section, we divide the Hurwitz-Lerch multiple zeta function as follows:
\begin{align*}
 & \zeta(s_{1},\ldots,s_{r};a_{1,}\ldots,a_{r};z_{1},\ldots,z_{r})\\
 & =\frac{1}{\Gamma(s_{1})\cdots\Gamma(s_{r})}(X_{1}+X_{2})\\
 & =\frac{1}{\Gamma(s_{1})\cdots\Gamma(s_{r})}(Y_{1}+Y_{2}+X_{2})\\
 & =\frac{1}{\Gamma(s_{1})\cdots\Gamma(s_{r})}\biggl(Y_{1}+\sum_{i_{2},\dots,i_{r}\in\{0,1/2\}}\widetilde{Y}_{2}(i_{2},\dots,i_{r})+\sum_{i_{2},\dots,i_{r}\in\{0,1/2\}}\widetilde{X}_{2}(i_{2},\dots,i_{r})\biggr).
\end{align*}

\subsection{Calculations for $Y_1$}
We consider $Y_{1}$, first. From \eqref{Y1}, we have
\begin{align*}
\frac{Y_{1}}{\Gamma(s_{1})\cdots\Gamma(s_{r})} & =\sum_{0\le k\le N}\sum_{n_{1}+\cdots+n_{r}=k,n_{1},\dots,n_{r}\ge0}(-1)^k\frac{B_{n_{1}}(a_1;z_1)\cdots B_{n_{r}}(a_{r};z_{r})}{n_{1}!\cdots n_{r}!}\\
 & \quad\cdot\frac{c^{k+s_{1}+\cdots+s_{r}-r}}{k+s_{1}+\cdots+s_{r}-r}\cdot\frac{1}{\Gamma(s_{r})}\\
 & \quad\cdot\frac{\Gamma(n_{2}+\cdots+n_{r}+s_{2}+\cdots+s_{r}-(r-1))}{\Gamma(n_{2}+\cdots+n_{r}+s_{1}+s_{2}+\cdots+s_{r}-(r-1))}\\
 & \quad\cdot\frac{\Gamma(n_{3}+\cdots+n_{r}+s_{3}+\cdots+s_{r}-(r-2))}{\Gamma(n_{3}+\cdots+n_{r}+s_{2}+\cdots+s_{r}-(r-2))}\cdots\\
 & \quad\cdot\frac{\Gamma(n_{r}+s_{r}-1)}{\Gamma(n_{r}+s_{r-1}+s_{r}-1)}.
\end{align*}
For $i=1,\dots,r$, let $s_{i}=-l_{i}+\epsilon_{i}$.
Putting 
\begin{align*}
H(l_{1},\dots,l_{r};n_{1},\dots,n_{r};\epsilon_{1},\dots,\epsilon_{r}) & :=\frac{c^{k-l(1,r)+\epsilon(1,r)-r}}{k-l(1,r)+\epsilon(1,r)-r}\cdot\frac{1}{\Gamma(-l_{r}+\epsilon_{r})}\\
 & \quad\cdot\frac{\Gamma(n(2,r)-l(2,r)+\epsilon(2,r)-(r-1))}{\Gamma(n(2,r)-l(1,r)+\epsilon(1,r)-(r-1))}\\
 & \quad\cdot\frac{\Gamma(n(3,r)-l(3,r)+\epsilon(3,r)-(r-2))}{\Gamma(n(3,r)-l(2,r)+\epsilon(2,r)-(r-2))}\\
 & \quad\cdots\\
 & \quad\cdot\frac{\Gamma(n_{r}-l_{r}+\epsilon_{r}-1)}{\Gamma(n_{r}-l_{r-1}-l_{r}+\epsilon_{r-1}+\epsilon_{r}-1)},
\end{align*}
we have
\begin{align}
&\frac{Y_{1}}{\Gamma(s_{1})\cdots\Gamma(s_{r})} \label{eq:Y1}\\
 &=\sum_{0\le k\le N}\sum_{n_{1}+\cdots+n_{r}=k,n_{1},\dots,n_{r}\ge0}(-1)^k\frac{B_{n_{1}}(a_1;z_1)\cdots B_{n_{r}}(a_{r};z_{r})}{n_{1}!\cdots n_{r}!}H(l_{1},\dots,l_{r};n_{1},\dots,n_{r};\epsilon_{1},\dots,\epsilon_{r}).\nonumber
\end{align}
Since
\begin{align*}
\frac{\Gamma(n(i+1,r)-l(i+1,r)+\epsilon(i+1,r)-(r-i))}{\Gamma(n(i+1,r)-l(i,r)+\epsilon(i,r)-(r-i))} & =\frac{(\epsilon(i+1,r))_{n(i+1,r)-l(i+1,r)-(r-i)}}{(\epsilon(i,r))_{n(i+1,r)-l(i,r)-(r-i)}}\cdot\frac{\Gamma(\epsilon(i+1,r))}{\Gamma(\epsilon(i,r))},
\end{align*}
we have
\begin{align*}
H(l_{1},\dots,l_{r};n_{1},\dots,n_{r};\epsilon_{1},\dots,\epsilon_{r}) & =\frac{c^{k-l(1,r)+\epsilon(1,r)-r}}{k-l(1,r)+\epsilon(1,r)-r}\cdot\frac{1}{(\epsilon_{r})_{-l_{r}}}\cdot\frac{1}{\Gamma(\epsilon(1,r))}\\
 & \quad\cdot\frac{(\epsilon(2,r))_{n(2,r)-l(2,r)-(r-1)}}{(\epsilon(1,r))_{n(2,r)-l(1,r)-(r-1)}}\\
 & \quad\cdot\frac{(\epsilon(3,r))_{n(3,r)-l(3,r)-(r-2)}}{(\epsilon(2,r))_{n(3,r)-l(2,r)-(r-2)}}\\
 & \quad\cdots\\
 & \quad\cdot\frac{(\epsilon(r,r))_{n(r,r)-l(r,r)-1}}{(\epsilon(r-1,r))_{n(r,r)-l(r-1,r)-1}},
\end{align*}
where $(\epsilon)_{n}:=\Gamma(n+\epsilon)/\Gamma(\epsilon)$. Note that 
\begin{align*}
(\epsilon)_{-n} & =\frac{1}{(\epsilon-1)\cdots(\epsilon-n)}=\frac{(-1)^{n}}{n!}+O(|\epsilon|)\quad(n\ge0),\\
(\epsilon)_{n} & =(\epsilon)\cdots(\epsilon+n-1)=\epsilon(n-1)!+O(|\epsilon|^{2})\quad(n>0).
\end{align*}
Let $N:=l(1,r)+r$. For $k<N$, we have
\begin{align*}
  \frac{c^{k-l(1,r)+\epsilon(1,r)-r}}{k-l(1,r)+\epsilon(1,r)-r}\cdot\frac{1}{(\epsilon_{r})_{-l_{r}}}\cdot\frac{1}{\Gamma(\epsilon(1,r))}
 & =\frac{c^{k+\epsilon(1,r)-N}}{k+\epsilon(1,r)-N}\cdot\frac{1}{(\epsilon_{r})_{-l_{r}}}\cdot\frac{1}{\Gamma(\epsilon(1,r))}\\
 & =O(|\epsilon(1,r)|) .
\end{align*}
If $k=N$, we have
\begin{align*}
  \frac{c^{k-l(1,r)+\epsilon(1,r)-r}}{k-l(1,r)+\epsilon(1,r)-r}\cdot\frac{1}{(\epsilon_{r})_{-l_{r}}}\cdot\frac{1}{\Gamma(\epsilon(1,r))}
 & =\frac{c^{\epsilon(1,r)}}{(\epsilon_{r})_{-l_{r}}}\cdot\frac{1}{\Gamma(\epsilon(1,r)+1)}\\
 & =\displaystyle(-1)^{l_{r}}l_{r}!+O(|\epsilon_{r}|)+O(|\epsilon(1,r)|).
\end{align*}
Note that $D:=n(j+1,r)-l(j,r)-(r-j)\le n(j+1,r)-l(j+1,r)-(r-j)=:U$
for $j=1,\dots,r-1$. If $D\le U\le0$, we have
\begin{align*}
\frac{(\epsilon(j+1,r))_{n(j+1,r)-l(j+1,r)-(r-j)}}{(\epsilon(j,r))_{n(j+1,r)-l(j,r)-(r-j)}} & =\frac{(-1)^{l_{j}}(-(n(j+1,r)-l(j,r)-(r-j)))!}{(-(n(j+1,r)-l(j+1,r)-(r-j)))!}+O(|\epsilon(j+1,r)|)+O(|\epsilon(j,r)|).
\end{align*}
If $D\le0<U$, we have
\begin{align*}
\frac{(\epsilon(j+1,r))_{n(j+1,r)-l(j+1,r)-(r-j)}}{(\epsilon(j,r))_{n(j+1,r)-l(j,r)-(r-j)}} & =O(|\epsilon(j+1,r)|).
\end{align*}
If $0<D\le U$, since $\left|\varepsilon_{k}/\varepsilon(j,r)\right|\ll1$ 
for $j=1,\ldots,r$ and $k=j,\ldots,r$, we have
\begin{align*}
&\frac{(\epsilon(j+1,r))_{n(j+1,r)-l(j+1,r)-(r-j)}}{(\epsilon(j,r))_{n(j+1,r)-l(j,r)-(r-j)}} \\
 & =\frac{\epsilon(j+1,r)}{\epsilon(j,r)}\left(\frac{(n(j+1,r)-l(j+1,r)-(r-j)-1)!}{(n(j+1,r)-l(j,r)-(r-j)-1)!}+O(|\epsilon(j+1,r)|)+O(|\epsilon(j,r)|)\right)\\
 & =\frac{\epsilon(j+1,r)}{\epsilon(j,r)}\cdot\frac{(n(j+1,r)-l(j+1,r)-(r-j)-1)!}{(n(j+1,r)-l(j,r)-(r-j)-1)!}+O(|\epsilon(j+1,r)|).
\end{align*}

Therefore we have
\begin{align*}
&\frac{Y_{1}}{\Gamma(s_{1})\cdots\Gamma(s_{r})} \\
 & =(-1)^{l(1,r)+r}\sum_{d_{1},\dots,d_{r-1}\in\{0,1\}}\sum_{(n_{1},\dots,n_{r})\in S^{(d_{1},\dots,d_{r-1})}}\frac{B_{n_{1}}(a_1;z_1)\cdots B_{n_{r}}(a_{r};z_{r})}{n_{1}!\cdots n_{r}!}h^{(d_{1},\dots,d_{r-1})}(n_{1},\dots,n_{r})\\
&\quad+\sum_{j=1}^rO(|\epsilon(j,r)|).
\end{align*}

\subsection{Calculations for $\widetilde{Y_{2}}$ and $\widetilde{X_{2}}$ }

Recall that all possible simple poles of $\widetilde{Y}_{2}(i_{2},\dots,i_{r})$ and $\widetilde{X}_{2}(i_{2},\dots,i_{r})$ are  located on the region
\[
\bigcup_{p=0}^{n}\bigcup_{j=2}^{r}U_{j}^p,
\]
where
\[
U_{j}^p=\begin{cases}
\{(s_{1},\dots,s_{r})\mid s_{j}+\cdots+s_{r}=r-j-p\} & \text{if }i_{j}=0,\\
\{(s_{1},\dots,s_{r})\mid s_{j-1}=-p\} & \text{if }i_{j}=1/2.
\end{cases}
\]
For $(s_{1},\dots,s_{r})=(-l_1+\epsilon_1,\dots,-l_r+\epsilon_r)$, we have
\[
\widetilde{Y}_{2}(i_{2},\dots,i_{r})\times(s(2)-q(2))\times\cdots\times(s(r)-q(r))=O(1)
\]
where
\[
s(j)-q(j):=\begin{cases}
s_{j}+\cdots+s_{r}+(l_{j}+\cdots+l_{r}) & \text{if }i_{j}=0,\\
s_{j-1}+l_{j-1} & \text{if }i_{j}=1/2.
\end{cases}
\]
Thus we have
\begin{align*}
\frac{\widetilde{Y}_{2}(i_{2},\dots,i_{r})}{\Gamma(s_{1})\cdots\Gamma(s_{r})}
 & =O\Biggl(|\epsilon_1\cdots\epsilon_r|\cdot\prod_{\substack{2\le j\le r\\i_{j}=0}}\frac{1}{|\epsilon(j,r)|}\cdot\prod_{\substack{2\le j\le r\\i_{j}=1/2}}\frac{1}{|\epsilon_{j-1}|}\Biggr)\\
  & =O\Biggl(|\epsilon_r|\cdot\prod_{\substack{2\le j\le r\\i_{j}=0}}\frac{|\epsilon_{j-1}|}{|\epsilon(j,r)|}\Biggr).
\end{align*}
Let $j_1,\ldots,j_t$ be all indices satisfying $i_{j_1},\ldots,i_{j_t}=0$ and $j_1<\cdots<j_t$. Then we have
\begin{align*}
\frac{\widetilde{Y}_{2}(i_{2},\dots,i_{r})}{\Gamma(s_{1})\cdots\Gamma(s_{r})}
 &=O\left(|\epsilon_r|\cdot\prod_{u=1}^t\frac{|\epsilon_{j_u-1}|}{|\epsilon(j_u,r)|}\right)\\
 &=O\left(|\epsilon_{j_1-1}|\cdot\prod_{u=1}^{t-1}\frac{|\epsilon_{j_{u+1}-1}|}{|\epsilon(j_u,r)|}\cdot\frac{|\epsilon_{r}|}{|\epsilon(j_t,r)|}\right)\\
  &=O\left(|\epsilon_{j_1-1}|\right)
\end{align*}
since $\left|\varepsilon_{k}/\varepsilon(j,r)\right|\ll1$ 
for $j=1,\ldots,r$ and $k=j,\ldots,r$. In a similar way, we can also estimate
\begin{align*}
\frac{\widetilde{X}_{2}(i_{2},\dots,i_{r})}{\Gamma(s_{1})\cdots\Gamma(s_{r})}
  &=O\left(|\epsilon_{j_1-1}|\right).
\end{align*}

\section{Appendix}
Here, we shall give some examples. Put $\bm{a}=(a_1,\ldots,a_r)$, $\bm{z}=(z_1,\ldots,z_r)$, and
\[
B_{(n_1,\ldots,n_r)}(\bm{a};\bm{z}):=\prod_{j=1}^rB_{n_j}(a_j;z_j)
\]
 for simplicity.

\begin{ex} When $r=2$, we have
\begin{align*}
\zeta(&-1+\epsilon_1,\epsilon_2;a_1,a_2;z_1,z_2) \\
&\qquad=\frac{1}{2} B_{(2,1)}(\bm{a};\bm{z})
+\frac{1}{3} B_{(3,0)}(\bm{a};\bm{z})-\frac{1}{6} B_{(0,3)}(\bm{a};\bm{z}) \frac{\epsilon_{2}}{\epsilon_{1}+\epsilon_{2}}+\sum_{j=1}^{2}O(|\epsilon_{j}|), \\
\zeta(&\epsilon_1,-1+\epsilon_2;a_1,a_2;z_1,z_2) \\
&\qquad=\frac{1}{2} B_{(2,1)}(\bm{a};\bm{z})
+\frac{1}{2} B_{(1,2)}(\bm{a};\bm{z})
+\frac{1}{6} B_{(3,0)}(\bm{a};\bm{z}) +\frac{1}{6} B_{(0,3)}(\bm{a};\bm{z}) \frac{\epsilon_{2}}{\epsilon_{1}+\epsilon_{2}}+\sum_{j=1}^{2}O(|\epsilon_{j}|), \\
 \zeta(&-1+\epsilon_1,-1+\epsilon_2;a_1,a_2;z_1,z_2) \\
&\qquad=\frac{1}{4} B_{(2,2)}(\bm{a};\bm{z})
+\frac{1}{3} B_{(3,1)}(\bm{a};\bm{z})
+\frac{1}{8} B_{(4,0)}(\bm{a};\bm{z}) -\frac{1}{24} B_{(0,4)}(\bm{a};\bm{z})  \frac{\epsilon_{2}}{\epsilon_{1}+\epsilon_{2}}+\sum_{j=1}^{2}O(|\epsilon_{j}|),
\end{align*}
where the Apostol-Bernoulli polynomials for $0\le n\le 4$ are as follows:
 \begin{align*}
&B_{0}(a;z) =0,
\qquad B_{1}(a;z) =\frac{1}{z-1},
\qquad B_{2}(a;z) =\frac{2}{z-1}a-\frac{2 z}{(z-1)^2},\\
&B_{3}(a;z) =\frac{3}{z-1}a^2-\frac{6 z}{(z-1)^2}a+\frac{3 z(z+1)}{(z-1)^3},\\
& B_{4}(a;z) =\frac{4}{z-1}a^3-\frac{12 z }{(z-1)^2}a^2+\frac{12  z(z+1)}{(z-1)^3}a-\frac{4 z(z^2+4z+ 1)}{(z-1)^4}.
\end{align*}
\end{ex}

\begin{ex}When $r=3$, we have
\begin{align*}
&\zeta(\epsilon_1,\epsilon_2,\epsilon_3;a_1,a_2,a_3;z_1,z_2,z_3) \\
&=-B_{(1,1,1)}(\bm{a};\bm{z}) 
-\frac{1}{2} B_{(2,0,1)}(\bm{a};\bm{z}) 
-\frac{1}{2} B_{(2,1,0)}(\bm{a};\bm{z})  
-\frac{1}{2} B_{(1,2,0)}(\bm{a};\bm{z}) 
-\frac{1}{6} B_{(3,0,0)}(\bm{a};\bm{z})\\
&\quad -\frac{1}{2} B_{(1,0,2)}(\bm{a};\bm{z}) \frac{\epsilon_{3}}{\epsilon_{2}+\epsilon_{3}}  -\biggl(\frac{1}{2} B_{(0,2,1)}(\bm{a};\bm{z})  +\frac{1}{6} B_{(0,3,0)}(\bm{a};\bm{z}) \biggr)\frac{\epsilon_{2}+\epsilon_3}{\epsilon_{1}+\epsilon_{2}+\epsilon_{3}}\\
&\quad -\biggl(\frac{1}{2} B_{(0,1,2)}(\bm{a};\bm{z}) +\frac{1}{6} B_{(0,0,3)}(\bm{a};\bm{z}) \biggr)\frac{\epsilon_{3}}{\epsilon_{1}+\epsilon_{2}+\epsilon_{3}}+\sum_{j=1}^{3}O(|\epsilon_{j}|), 
\end{align*}
\begin{align*}
&\zeta(-1+\epsilon_1,\epsilon_2,\epsilon_3;a_1,a_2,a_3;z_1,z_2,z_3)\\
&=-\frac{1}{2} B_{(2,1,1)}(\bm{a};\bm{z}) 
-\frac{1}{4} B_{(2,2,0)}(\bm{a};\bm{z}) 
-\frac{1}{3} B_{(3,1,0)}(\bm{a};\bm{z}) 
 -\frac{1}{3} B_{(3,0,1)}(\bm{a};\bm{z}) 
-\frac{1}{8} B_{(4,0,0)}(\bm{a};\bm{z}) \\
&\quad -\frac{1}{4} B_{(2,0,2)}(\bm{a};\bm{z})  \frac{\epsilon_{3}}{\epsilon_{2}+\epsilon_{3}} 
 +\biggl(\frac{1}{6} B_{(0,3,1)}(\bm{a};\bm{z})  +\frac{1}{24} B_{(0,4,0)}(\bm{a};\bm{z}) \biggr) \frac{\epsilon_{2}+\epsilon_{3}}{\epsilon_{1}+\epsilon_{2}+\epsilon_{3}} \\
&\quad +\biggl(\frac{1}{4} B_{(0,2,2)}(\bm{a};\bm{z}) 
 +\frac{1}{6} B_{(0,1,3)}(\bm{a};\bm{z})  
 +\frac{1}{24} B_{(0,0,4)}(\bm{a};\bm{z}) \biggr)  \frac{\epsilon_{3}}{\epsilon_{1}+\epsilon_{2}+\epsilon_{3}}+\sum_{j=1}^{3}O(|\epsilon_{j}|) ,
\end{align*}
\begin{align*}
&\zeta(\epsilon_1,-1+\epsilon_2,\epsilon_3;a_1,a_2,a_3;z_1,z_2,z_3) \\
&=-\frac{1}{2} B_{(2,1,1)}(\bm{a};\bm{z}) 
 -\frac{1}{2} B_{(2,2,0)}(\bm{a};\bm{z}) -\frac{1}{3} B_{(3,1,0)}(\bm{a};\bm{z}) 
 -\frac{1}{6} B_{(3,0,1)}(\bm{a};\bm{z}) 
 -\frac{1}{2} B_{(1,2,1)}(\bm{a};\bm{z}) \\
&\quad 
 -\frac{1}{3} B_{(1,3,0)}(\bm{a};\bm{z}) -\frac{1}{12} B_{(4,0,0)}(\bm{a};\bm{z}) \\
&\quad +\frac{1}{6} B_{(1,0,3)}(\bm{a};\bm{z})  \frac{\epsilon_{3}}{\epsilon_{2}+\epsilon_{3}} 
-\biggl(\frac{1}{6} B_{(0,3,1)}(\bm{a};\bm{z}) + \frac{1}{12} B_{(0,4,0)}(\bm{a};\bm{z})  \biggr)\frac{\epsilon_{2}+\epsilon_{3}}{\epsilon_{1}+\epsilon_{2}+\epsilon_{3}} \\
&\quad +\biggl(\frac{1}{6} B_{(0,1,3)}(\bm{a};\bm{z})  +\frac{1}{12} B_{(0,0,4)}(\bm{a};\bm{z}) \biggr) \frac{\epsilon_{3}}{\epsilon_{1}+\epsilon_{2}+\epsilon_{3}}+\sum_{j=1}^{3}O(|\epsilon_{j}|),
\end{align*}
\begin{align*}
&\zeta(\epsilon_1,\epsilon_2,-1+\epsilon_3;a_1,a_2,a_3;z_1,z_2,z_3) \\
&=-\frac{1}{2} B_{(2,1,1)}(\bm{a};\bm{z}) 
 -\frac{1}{4} B_{(2,2,0)}(\bm{a};\bm{z}) 
 -\frac{1}{4} B_{(2,0,2)}(\bm{a};\bm{z})
 -\frac{1}{2} B_{(1,2,1)}(\bm{a};\bm{z}) 
 -\frac{1}{2} B_{(1,1,2)}(\bm{a};\bm{z}) \\
&\quad -\frac{1}{6} B_{(3,1,0)}(\bm{a};\bm{z}) 
 -\frac{1}{6} B_{(3,0,1)}(\bm{a};\bm{z}) 
 -\frac{1}{6} B_{(1,3,0)}(\bm{a};\bm{z}) 
 -\frac{1}{24} B_{(4,0,0)}(\bm{a};\bm{z}) \\
&\quad -\frac{1}{6} B_{(1,0,3)}(\bm{a};\bm{z}) \frac{\epsilon_{3}}{\epsilon_{2}+\epsilon_{3}} \\
&\quad -\biggl(\frac{1}{4} B_{(0,2,2)}(\bm{a};\bm{z}) +\frac{1}{6} B_{(0,3,1)}(\bm{a};\bm{z}) +\frac{1}{24} B_{(0,4,0)}(\bm{a};\bm{z}) \biggr) \frac{\epsilon_{2}+\epsilon_{3}}{\epsilon_{1}+\epsilon_{2}+\epsilon_{3}} \\
&\quad -\biggl(\frac{1}{6} B_{(0,1,3)}(\bm{a};\bm{z}) +\frac{1}{24} B_{(0,0,4)}(\bm{a};\bm{z}) \biggr) \frac{\epsilon_{3}}{\epsilon_{1}+\epsilon_{2}+\epsilon_{3}} +\sum_{j=1}^{3}O(|\epsilon_{j}|).
\end{align*}
\end{ex}

\begin{ex}When $r=4$, we have
\begin{align*}
&\zeta(\epsilon_1,\epsilon_2,\epsilon_3,\epsilon_4;a_1,a_2,a_3,a_4;z_1,z_2,z_3,z_4) \\
&=B_{(1,1,1,1)}(\bm{a};\bm{z}) 
 +\frac{1}{2} B_{(2,1,1,0)}(\bm{a};\bm{z}) 
 +\frac{1}{2} B_{(2,1,0,1)}(\bm{a};\bm{z})
 +\frac{1}{2} B_{(2,0,1,1)}(\bm{a};\bm{z})
 +\frac{1}{2} B_{(1,2,1,0)}(\bm{a};\bm{z}) \\
&\quad 
 +\frac{1}{2} B_{(1,2,0,1)}(\bm{a};\bm{z})
 +\frac{1}{4} B_{(2,2,0,0)}(\bm{a};\bm{z})
 +\frac{1}{2} B_{(1,1,2,0)}(\bm{a};\bm{z})
 +\frac{1}{4} B_{(2,0,2,0)}(\bm{a};\bm{z})
 +\frac{1}{6} B_{(3,1,0,0)}(\bm{a};\bm{z})\\
&\quad 
 +\frac{1}{6} B_{(3,0,1,0)}(\bm{a};\bm{z})
 +\frac{1}{6} B_{(3,0,0,1)}(\bm{a};\bm{z})
 +\frac{1}{6} B_{(1,3,0,0)}(\bm{a};\bm{z})
 +\frac{1}{24} B_{(4,0,0,0)}(\bm{a};\bm{z}) \\
&\quad +\biggl(
  \frac{1}{2} B_{(1,1,0,2)}(\bm{a};\bm{z})
 +\frac{1}{4} B_{(2,0,0,2)}(\bm{a};\bm{z})
  \biggr) \frac{\epsilon_{4}}{\epsilon_{3}+\epsilon_{4}} \\
&\quad+\biggl(
   \frac{1}{2} B_{(1,0,2,1)}(\bm{a};\bm{z})
  +\frac{1}{6} B_{(1,0,3,0)}(\bm{a};\bm{z})
   \biggr) \frac{\epsilon_{3}+\epsilon_{4}}{\epsilon_{2}+\epsilon_{3}+\epsilon_{4}}  \\
&\quad+\biggl(
 \frac{1}{2} B_{(1,0,1,2)}(\bm{a};\bm{z}) 
 +\frac{1}{6} B_{(1,0,0,3)}(\bm{a};\bm{z})
 \biggr) \frac{\epsilon_{4}}{\epsilon_{2}+\epsilon_{3}+\epsilon_{4}} \\
&\quad +\biggl(
 \frac{1}{2} B_{(0,2,1,1)}(\bm{a};\bm{z})
  +\frac{1}{4} B_{(0,2,2,0)}(\bm{a};\bm{z}) 
  +\frac{1}{6} B_{(0,3,0,1)}(\bm{a};\bm{z})
  +\frac{1}{6} B_{(0,3,1,0)}(\bm{a};\bm{z}) \\
  &\qquad\qquad
  +\frac{1}{24} B_{(0,4,0,0)}(\bm{a};\bm{z}) 
 \biggr)  \frac{\epsilon_{2}+\epsilon_{3}+\epsilon_{4}}{\epsilon_{1}+\epsilon_{2}+\epsilon_{3}+\epsilon_{4}} \\
&\quad +\biggl(
 \frac{1}{2} B_{(0,1,2,1)}(\bm{a};\bm{z}) 
 +\frac{1}{6} B_{(0,1,3,0)}(\bm{a};\bm{z}) 
 +\frac{1}{6} B_{(0,0,3,1)}(\bm{a};\bm{z}) 
 +\frac{1}{24} B_{(0,0,4,0)}(\bm{a};\bm{z}) 
  \biggr)  \frac{\epsilon_{3}+\epsilon_{4}}{\epsilon_{1}+\epsilon_{2}+\epsilon_{3}+\epsilon_{4}} \\
&\quad +\biggl(
 \frac{1}{2} B_{(0,1,1,2)}(\bm{a};\bm{z}) 
 +\frac{1}{4} B_{(0,0,2,2)}(\bm{a};\bm{z}) 
 +\frac{1}{6} B_{(0,1,0,3)}(\bm{a};\bm{z}) 
 +\frac{1}{6} B_{(0,0,1,3)}(\bm{a};\bm{z}) \\
  &\qquad\qquad
 +\frac{1}{24} B_{(0,0,0,4)}(\bm{a};\bm{z}) 
  \biggr)  \frac{\epsilon_{4}}{\epsilon_{1}+\epsilon_{2}+\epsilon_{3}+\epsilon_{4}} \\
&\quad +\frac{1}{4} B_{(0,2,0,2)}(\bm{a};\bm{z}) 
 \frac{ \epsilon_{4}(\epsilon_{2}+\epsilon_{3} +\epsilon_{4}) }{(\epsilon_{3}+\epsilon_{4}) (\epsilon_{1}+\epsilon_{2}+\epsilon_{3}+\epsilon_{4})}.
\end{align*}
\end{ex}

\section*{Acknowledgements}
This work was supported by JSPS KAKENHI Grant Number JP19K14511.


\end{document}